\documentclass{amsart}
\title{Inductive Formulas for some Arithmetic Functions}
\usepackage{amssymb,amsmath,amsthm,epsfig,graphics,latexsym}
 \theoremstyle{definition}
 \newtheorem{definition}{Definition}
  \theoremstyle{plain}
  \newtheorem{lemma}      [definition]{Lemma}
  
  \newtheorem{theorem}    {Theorem}

  \theoremstyle{remark}

\begin{document}
  \author{Mohamed El Bachraoui}
  \address{Dept. Math. Sci,
 United Arab Emirates University, PO Box 17551, Al-Ain, UAE}
 \email{melbachraoui@uaeu.ac.ae}
  %
  \begin{abstract}
  We prove recursive formulas involving sums of divisors and sums of triangular numbers
   and give a variety of identities
  relating arithmetic functions to divisor functions providing inductive identities
  for such arithmetic functions.
  \end{abstract}
  \date{\textit{\today}}
  \maketitle
 \section{Introduction}
 In this note we give a natural extension of a theorem given by Apostol in his book \cite{Apostol}
 on analytic number theory and we use
 it to deduce a variety of formulas involving sums of divisors functions.
 Let $\mathbb{N}=\{1,2,3,\ldots\}$ and $\mathbb{N}_0 = \{0,1,2,\ldots\}$.
 \begin{theorem}\label{one-product} \cite[Theorem 14.8]{Apostol}
 Let $A \subseteq {\mathbb{N}}$ and let
 $f : A \to \mathbb{C}$ be an arithmetic function such that both
 \[
  F_{A} (x) =
  \prod_{n \in A}(1- x^{n})^{-\frac{f(n)}{n}} =
  \sum_{n=0}^{\infty} p_{A,f}(n) x^n
 \]
 and
 \[
 G_{A} (x) =
  \sum_{n \in A} \frac{f(n)}{n} x^{n}
 \]
 converge absolutely and represent analytic functions in the unit disk $|x|<1$.
 Then
 \[
 n p_{A,f}(n) =
 \sum_{k=1}^n \left(p_{A,f}(n-k) f_{A}(k)\right),
 \]
 where
  $p_{A,f}(0)=1$ and
  \[ f_{A}(k) = \sum_{\substack{d|k \\ d\in A}}f(d). \]
 \end{theorem}

 \noindent
 A direct consequence of Theorem \ref{one-product} states that
 \[ n p(n) = \sum_{k=1}^n \sigma(k) p (n-k), \]
 where $p(n)$ is the partition function and $\sigma(n)$ is the sum of divisors function, see Apostol \cite{Apostol}.
 An important argument to deduce the last identity is the fact that the generating function for $p(n)$ is
 \[ \prod_{n=1}^{\infty} (1-x^n)^{-1}. \]
 Theorem \ref{one-product} has also been used by Robbins in \cite{Robbins} to give formulas relating different
 arithmetic functions.
 However, for many arithmetic functions having generating functions of the form
 \[
 \prod_{i=1}^{\alpha} \prod_{n \in A_i}(1- x^{n})^{-\frac{f_i(n)}{n}}
 \]
 with multiple infinite products involved,
 this theorem does not apply without appeal to the generating functions
 of the individual infinite products.
  The main motivation of this work is to deal with such arithmetic functions in a direct way.

  \noindent

  We will need the following three identities due to Jacobi, Gauss, and Ramanujan respectively:
 \begin{equation} \label{Jacobi}
  \prod_{n=1}^{\infty}(1-x^{2n})(1-x^{2n-1})^2 = 1 + 2 \sum_{n=1}^{\infty}(-1)^n x^{n^2}.
  \end{equation}
  \begin{equation} \label{Gauss}
  \prod_{n=1}^{\infty}(1-x^{2n})(1-x^{2n-1})^{-1} = \sum_{n=0}^{\infty} x^{\frac{n(n+1)}{2}}.
  \end{equation}
  \begin{equation} \label{Ramanujan-identity}
  x \left( \sum_{n=0}^{\infty} x^{\frac{n(n+1)}{2}} \right)^8 =
  x \prod_{n=1}^{\infty}(1-x^{2n})^8 (1-x^{2n-1})^{-8} =
   \sum_{n=1}^{\infty}\frac{n^3 x^{n}}{1-x^{2n}}.
  \end{equation}
  Identities (\ref{Jacobi}) and (\ref{Gauss}) can be found in Hardy and Wright \cite{Hardy-Wright} and
  identity (\ref{Ramanujan-identity}) can be found in Ramanujan \cite{Ramanujan} with a proof in Ewell \cite{Ewell}.
  Note that the infinite sum
  \[
  \left(\sum_{n=0}^{\infty} x^{\frac{n(n+1)}{2}} \right)^8
  \]
  which appears in formula (\ref{Ramanujan-identity}) counts the number of representations of a positive integer
  as the sum of eight triangular numbers, see Definition \ref{triangular} below for triangular numbers.
  For an account on representations
  as sums of triangular numbers see for instance Ono et al \cite{Ono-Robins-Wahl}.
   We shall also require the following two identities due to Rogers and Ramanujan, see reference \cite{Hardy-Wright}.
  \begin{equation}\label{Rogers-Ramanujan}
  \begin{split}
  \sum_{n=0}^{\infty} R_1(n) x^n &= \prod_{n=1}^{\infty}\left((1-x^{5n-1}) (1-x^{5n-4})\right)^{-1} =
  1 + \sum_{n=1}^{\infty} \frac{x^{n^2}}{\prod_{j=1}^n(1-x^j)}. \\
  \sum_{n=0}^{\infty} R_2(n) x^n &= \prod_{n=1}^{\infty} \left((1-x^{5n-2})(1-x^{5n-3}) \right)^{-1} =
  1 + \sum_{n=1}^{\infty} \frac{x^{n(n+1)}}{\prod_{j=1}^n(1-x^j)}.
  \end{split}
  \end{equation}
 \section{Main Result}
 \begin{theorem}\label{main1}
 Let $\alpha \in \mathbb{N}$. Let $A_1, A_2, \ldots A_{\alpha} \subseteq {\mathbb{N}}$ and
 $\underline{A}=(A_1,A_2,\ldots,A_{\alpha})$ and let
 $f_i : A_i \to \mathbb{C}$ for $i=1,2,\ldots,\alpha$ be arithmetic functions and let
 $\underline{f}=(f_1,f_2,\ldots,f_{\alpha})$.
 If both
 \[
  F_{\underline{A}} (x) =
  \prod_{i=1}^{\alpha} \prod_{n \in A_i}(1- x^{n})^{-\frac{f_i(n)}{n}} =
  \sum_{n=0}^{\infty} p_{\underline{A},\underline{f}}(n) x^n
 \]
 and
 \[
  \sum_{i=1}^{\alpha} \sum_{n \in A_i} \frac{f_i(n)}{n} x^{n}
 \]
 converge absolutely and represent analytic functions in the unit disk $|x|<1$,
 then
 \begin{equation*}\label{main-formula}
 n p_{\underline{A},\underline{f}}(n) =
 \sum_{k=1}^n \left(p_{\underline{A},\underline{f}}(n-k) \sum_{i=1}^{\alpha} f_{i,A_i}(k)\right),
 \end{equation*}
 where
  $p_{\underline{A},\underline{f}}(0)=1$ and
  \[ f_{i,A_i}(k) = \sum_{\substack{d \mid k \\ d\in A_i}}f_i(d). \]
 \end{theorem}
 \begin{proof}
 We have
 \[
 \begin{split}
 \log F_{\underline{A}}(x) &=
 \sum_{i=1}^{\alpha} - \sum_{n\in A_i} \frac{f_i (n)}{n} \log(1-x^n) \\
 &=
 \sum_{i=1}^{\alpha} \sum_{n\in A_i} \frac{f_i (n)}{n}\sum_{m=1}^{\infty}\frac{x^{mn}}{m}.
 \end{split}
 \]
 Differentiating and multiplying by $x$ gives
 \[
 x \frac{F_{\underline{A}}'(x)}{F_{\underline{A}}(x)} =
 \sum_{m=1}^{\infty} \sum_{i=1}^{\alpha} \sum_{n\in A_i} f_i(n)x^{m n}
 =
 \sum_{i=1}^{\alpha} \sum_{k=1}^{\infty} f_{i,A_i}(k) x^k.
 \]
 Then
 \[
 x F_{\underline{A}}'(x) = F_{\underline{A}}(x) ( \sum_{i=1}^{\alpha} \sum_{k=1}^{\infty} f_{i,A_i}(k) x^k ), \]
 and as
 \[
 x F_{\underline{A}}'(x) = \sum_{n=1}^{\infty} n p_{\underline{A},\underline{f}}(n) x^n,
 \]
 the result follows by matching coefficients.
 \end{proof}
 \section{Applications to divisor functions}
 \noindent
 \begin{definition}
 Let $q \in \mathbb{Q}$ and let 
 \[
 \sigma(q) = \begin{cases}
 \sum_{d\mid q}d\ \text{if\ } q \in \mathbb{N}, \\
 1\quad \text{if\ } q=0, \\
 0\quad \text{if\ } q\in \mathbb{Q}\setminus \mathbb{N}_0.
 \end{cases}
 \]
 \end{definition}
 \begin{definition}
 Let $n, r\in \mathbb{N}_0$, let $m \in \mathbb{N}$, and let
 \[
 \sigma_{r,m}(n) = \sum_{\substack{d\mid n\\ d\equiv r\bmod m}} d
 \]
 If $m=2$ and $r=1$ we shall write $\sigma^{o}(n)$ rather than $\sigma_{1,2}(n)$ and
 if $m=2$ and $r=0$ we shall write $\sigma^{E}(n)$ rather than $\sigma_{0,2}(n)$.
 \end{definition}
  \begin{definition} \label{square}
 Let the function $s$ be defined on $\mathbb{N}_0$ by
 \[
 s(n) = \begin{cases}
 1\quad \text{if\ } n=m^2, \\
 0\quad \text{otherwise.}
 \end{cases}
 \]
 \end{definition}
 \begin{theorem} \label{sigma-inductive}
 If $n\in \mathbb{N}$, then
 \[
  (-1)^n s(n) n = \frac{-\sigma (n)- \sigma^{o}(n)}{2} + \sum_{k\geq 1} (-1)^{k+1} (\sigma (n-k^2)+ \sigma^{o}(n-k^2)).
 \]
 \end{theorem}
 \begin{proof}
 Let $A_1$ be the set of even nonnegative integers, let $A_2$ be the set of odd nonnegative integers, and let
 $f_1(n)=-n$ and $f_2(n)=-2n$ be defined on $A_1$ and $A_2$ respectively. Then
 \[
  F_{\underline{A}}(x) = \prod_{n \in A_1} (1-x^n) \prod_{n\in A_2} (1-x^n)^2 =
  \prod_{n=1}^{\infty} (1-x^{2n})(1- x^{2n-1})^2.
  \]
 So
 by identity (\ref{Jacobi}) we have
 \[
  F_{\underline{A}}(x) = 1 + \sum_{n=1}^{\infty}p_{\underline{A},\underline{f}}(n) x^n =
   1 + 2 \sum_{n=1}^{\infty} (-1)^n x^{n^2},
 \]
 and therefore for $n\in \mathbb{N}$
 \[
 p_{\underline{A},\underline{f}}(n) = 2 (-1)^n s(n).
 \]
 Moreover, we have
 \[ f_{1,A_1}(k)= - \sigma^{E} (k),\  f_{2,A_2}(k)= -2 \sigma^{o} (k),\ \text{and so\ }
 f_{1,A_1}(k) + f_{2,A_2}(k)= - \sigma(k) - \sigma^{o}(k). \]
 Putting in Theorem \ref{main1} we find
 \[
 -(\sigma(n) + \sigma^{o} (n)) + \sum_{k=1}^{n-1} 2 (-1)^k \left( - \sigma (n-k^2) - \sigma^{o}(n-k^2) \right) =
 2 n (-1)^n s(n),
 \]
 and the result follows.
 \end{proof}
 We note that Theorem \ref{sigma-inductive} has been given in Liouville \cite{Liouville} and it is provable
  by Liouville's elementary methods, see Williams \cite[Theorem 6.2]{Williams}.
 \begin{definition} \label{triangular}
 Let $n\in\mathbb{N}_0$ and let $T(n)= \frac{n(n+1)}{2}$. The number is
 called \emph{triangular} if $n=T(m)$ for some $m \in \mathbb{N}_0$. Further
 let the function $t$ be defined on $\mathbb{N}_0$ as follows:
 \[
 t(n) = \begin{cases}
 1\quad \text{if\ } n = T(m),\\
 0\quad \text{otherwise.}
 \end{cases}
 \]
 \end{definition}
 \begin{theorem}
 We have
 \[
  t(n)n = \sum_{k\geq 0} \sigma^{o}(n- T(k)) - \sigma^{E}(n-T(k)).
 \]
 \end{theorem}
 \begin{proof}
 Let $A_1$ be the set of even nonnegative integers and $f_1(n)=-n$ and let
 $A_2$ be the set of odd nonnegative integers and $f_2(n)=n$. Then
 by identity (\ref{Gauss})
 \[
 F_{\underline{A}}(x)= \prod_{n\in A_1} (1-x^n) \prod_{n\in A_2} (1-x^n)^{-1}
 =
 \prod_{n=1}^{\infty} (1-x^{2n}) (1- x^{2n-1})^{-1} = \sum_{n=0}^{\infty} x^{T(n)}.
 \]
 Writing $F_{\underline{A}}(x)= \sum_{n=0}^{\infty} p_{\underline{A},\underline{f}}(n) x^n$, we have
 \[
 p_{\underline{A},\underline{f}}(n) = t(n).
 \]
 Moreover, we have
 $f_{1,A_1}(k) = -\sigma^E(k)$ and $f_{2,A_2}(k) = \sigma^{o}(k)$. Then by Theorem \ref{main1}
 we find
 \[
 n p_{\underline{A},\underline{f}}(n) =
 \sum_{k=0}^{n-1} p_{\underline{A},\underline{f}}(k)\left( \sigma^{o}(n-k) - \sigma^{E}(n-k) \right)
 \]
 and thus
 \[
 \sum_{k=0}^{n-1} \left( \sigma^{o}(n-T(k)) - \sigma^{E}(n-T(k)) \right)
 = t(n) n.
 \]
 \end{proof}
 \section{Arithmetic functions connected to divisor functions}
 \begin{definition}
 Let $a(0) =1$ and
  \[
  a(n)= \sum_{\substack{d|n \\ d\equiv 1 \bmod 2}} \left( \frac{n}{d} \right)^3.
 \]
 \end{definition}
 \begin{lemma} \label{cubic-divisors}
 We have
 \[
 x \prod_{n=1}^{\infty} (1-x^{2n})^8 (1- x^{2n-1})^{-8} = \sum_{n=0}^{\infty} a(n) x^n.
 \]
 \end{lemma}
 \begin{proof}
 By equation (\ref{Ramanujan-identity}) we have
 \[
 x \left(\prod_{n=1}^{\infty} (1-x^{2n}) (1- x^{2n-1})^{-1}\right)^8 =
 x \left( \sum_{n=0}^{\infty} x^{T(n)} \right)^8 = \sum_{n=1}^{\infty} \frac{n^3 x^{n}}{1-x^{2n}}.
 \]
 Moreover, it can be verified that
 \[
 \sum_{n=1}^{\infty} \frac{n^3 x^{n}}{1-x^{2n}} = \sum_{n=0}^{\infty} a(n) x^n. \]
 This proves the result.
 \end{proof}
 \noindent
 The coefficients $a(n)$ are connected to the divisors functions as follows.
 \begin{theorem}
 We have
 \[
 n a(n)= 8 \sum_{k\geq 0} a(k) (\sigma^{o}(n-k) - \sigma^{E}(n-k)).
 \]
 \end{theorem}
 \begin{proof}
 By Lemma \ref{cubic-divisors}, we have
 \[
 x \prod_{n=1}^{\infty} (1-x^{2n})^8 (1- x^{2n-1})^{-8} = \sum_{n=0}^{\infty} a(n) x^n.
 \]
 Then Theorem \ref{main1} yields
 \[
 n a(n) = \sum_{n=0}^{n-1} a(k) \left( 8 \sigma^{o}(n-k) - 8 \sigma^{E}(n-k) \right),
 \]
 which gives the result.
 \end{proof}
  \begin{definition}
 A partition of $n$ is called \emph{$p$-regular} if its parts repeat less than $p$ times.
 The number of such partitions is denoted by $Q^{(p)}(n)$.
 \end{definition}

 \noindent
 The generating function for $Q^{(p)}(n)$ is
 \begin{equation} \label{p-regular}
 \sum_{n=0}^{\infty}Q^{(p)}(n) x^n = \prod_{n=1}^{\infty} (1-q^n)(1- q^{pn})^{-1}.
 \end{equation}
 See Gordon and Ono \cite{Gordon-Ono} and Alladi \cite{Alladi} for details about this function.
 \begin{theorem}
 We have
 \[
  n Q^{(p)} (n) = \sum_{k\geq 0} Q^{(p)}(k) \left( \sigma_{0,p}(n-k) - \sigma (n-k) \right).
 \]
 \end{theorem}
 \begin{proof}
  By identity (\ref{p-regular})
  \[
  \sum_{n=0}^{\infty}Q^{(p)}(n) x^n =
  \prod_{n\in A_1} (1-q^n) \prod_{n\in A_2} (1-q^n)^{-1}, \]
  where $A_1= \mathbb{N}$ and $A_2$ is the set of positive multiples of $p$.
  Then by Theorem \ref{main1} we find
  \[
  n Q^{(p)}(n) = \sum_{k=0}^{n-1} Q^{(p)}(k) \left( \sigma_{0,p}(n-k) - \sigma (n-k) \right),
  \]
  as required.
  \end{proof}
 We now give inductive formulas for the coefficients $R_1(n)$ and $R_2(n)$ in the Rogers-Ramanujan identities (\ref{Rogers-Ramanujan}).
  \begin{theorem} \label{rog-ram-coeff}
  We have
  \[
  n R_1(n) = \sum_{k \geq 0} R_1(k) \left(\sigma_{1,5}(n-k)+ \sigma_{4,5}(n-k) \right)
  \]
  and
  \[
  n R_2(n) = \sum_{k \geq 0} R_2(k) \left(\sigma_{2,5}(n-k)+ \sigma_{3,5}(n-k) \right).
  \]
  \end{theorem}
  \begin{proof}
  Apply Theorem \ref{main1} to equations (\ref{Rogers-Ramanujan}).
  \end{proof}
  \section{Other product-to-sum identities}
 In this section we list some identities along with
 the corresponding results when we apply Theorem \ref{main1}. Straightforward verifications are left to the reader.
 A direct consequence of identity (\ref{Jacobi}) is
 \[
 \prod_{n=1}^{\infty} \frac{(1-x^{2n})^5}{(1-x^n)^2 (1-x^{4n})^2} = 1 + 2 \sum_{n=1}^{\infty} x^{n^2}, \]
 which by Theorem \ref{main1} gives
 \[
 n s(n) = \sigma(n)-5 \sigma(n/2) + 4 \sigma(n/4) +
 2 \sum_{k=1}^{n-1} s(n-k)\left(\sigma(k)-5 \sigma(k/2)+ 4 \sigma(k/4)\right).
 \]
 \begin{definition} \label{sum-triangular}
 Form $m\in \mathbb{N}$ and $n\in \mathbb{N}_0$ let $\delta_m (n)$ be the number of representations of
 $n$ as a sum of $m$ triangular numbers.
 \end{definition}
 Examples of formulas for $\delta_m(n)$ for a variety of cases of $m$ are found in Ono et al \cite{Ono-Robins-Wahl}. For instance, if
 $m \in \{1, 2,6,10\} \cup\{4 l:\ l\in \mathbb{N}\}$ the authors' results imply
 \[
 \sum_{n=0}^{\infty} \delta_{m}(n) x^n =
 \prod_{n=1}^{\infty} (1-x^{2n})^{2m} (1-x^n)^{-m},
 \]
 which by virtue of Theorem \ref{main1} translates into
 \[
 n \delta_m (n) = m \sum_{k=1}^n \left( \sigma^o (k)-\sigma^E (k) \right) \delta_m(n-k).
 \]

\noindent{\bf Acknowledgment.} The author is grateful to the referee
 for valuable comments and interesting suggestions.
 
%
\end{document}